\documentclass[a4paper,10pt,reqno]{amsart}

\usepackage{graphicx}
\usepackage{amsmath}
\usepackage{amsthm}
\usepackage{amssymb}
\usepackage{enumitem}
\usepackage{hyperref}
\numberwithin{equation}{section}
\setlist{leftmargin=*,label={\rm(\arabic*)}}

\theoremstyle{plain}

\newtheorem{thm}{Theorem}[section]
\newtheorem{prop}[thm]{Proposition}
\newtheorem{cor}[thm]{Corollary}
\newtheorem{lem}[thm]{Lemma}
\theoremstyle{definition}

\newtheorem{exm}[thm]{Example}
\theoremstyle{remark}

\DeclareMathOperator{\Hom}{Hom}
\DeclareMathOperator{\Ext}{Ext}
\DeclareMathOperator{\Tor}{Tor}

\DeclareMathOperator{\Ker}{Ker}
\DeclareMathOperator{\Coker}{Coker}

\DeclareMathOperator{\rad}{rad}

\DeclareMathOperator{\Mon}{Mon}

\newcommand{\xto}{\xrightarrow}
\newcommand{\lto}{\longrightarrow}
\mathchardef\mhyphen="2D
\renewcommand{\-}{\mhyphen}
\newcommand{\bo}{\mathrm{b}}

\newcommand{\op}{\mathrm{op}}

\newcommand{\ev}{\mathrm{ev}}

\renewcommand{\mod}{\mathrm{mod}}

\newcommand{\proj}{\mathrm{proj}}

\newcommand{\Gproj}{\mathrm{Gproj}}

\newcommand{\uGproj}{\underline{\mathrm{Gproj}}}

\newcommand{\bbZ}{\mathbb{Z}}

\newcommand{\bfC}{\mathbf{C}}
\newcommand{\bfD}{\mathbf{D}}

\newcommand{\bfK}{\mathbf{K}}

\newcommand{\X}{\mathcal{X}}

\begin{document}
%---------------title-------------------%
\title[A description of Gorenstein projective modules]
      {A description of Gorenstein projective modules over the tensor products of algebras}
\author[Dawei Shen]{Dawei Shen}
\address{Department of Mathematics \\
         Shanghai Key laboratory of PMMP \\
         East China Normal University \\
         Shanghai 200241 \\
         P. R. China}
\email{dwshen@math.ecnu.edu.cn}
\subjclass[2010]{Primary 16G10; Secondary 16D40}
\keywords{Gorenstein algebra, Gorenstein projective module}
%\dedicatory{}
%\commby{}
\date{\today}

%---------------abstract-------------------%
\begin{abstract}
Let $A$ be a coherent algebra
and $B$ be a finite-dimensional Gorenstein algebra over a field $k$.
We describe finitely presented Gorenstein projective
$A\otimes_k B$-modules in terms of their underlying onesided modules.
Moreover, if the  global dimension  of $B$ is finite,
we give a more precise description
of finitely presented Gorenstein projective $A\otimes_k B$-modules.
\end{abstract}
\maketitle

\section{Introduction} \label{sec:1}
The study of   Gorenstein projective modules
is originated by M.~Auslander and M.~Bridger in \cite{AB1969}
under the name  ``modules of G-dimension zero''.
The notion of Gorenstein projective modules
is introduced by E.~E.~Enochs and O.~M.~G.~Jenda  in \cite{EJ1995,EJ2000}
and is extensively studied  by their coauthors.
By the pioneering work of  R.-O.~Buchweitz in \cite{Buc1987},
the stable categories of Gorenstein projective modules
are closely related to singularity categories.

One of the most important tasks in Gorenstein homological algebra
is to describe Gorenstein projective modules.
In this paper,
we study Gorenstein projective modules
over the tensor product of two algebras
in terms of their underlying onesided modules.

Another motivation of this work is from  monomorphism categories.
G.~Birkhoff initiates the study of monomorphisms
between abelian groups \cite{Bir1935}.
C.~M.~Ringel and M.~Schmidmeier \cite{RS2006} investigate  submodule categories.
X.-H.~Luo and P.~Zhang \cite{LZ2013}  generalize their work
and introduce monomorphism categories over finite acyclic  quivers.

Let $A$ be a coherent algebra over a field $k$, and
let $Q$ be a finite acyclic quiver.
The monomorphism category $\Mon(Q, A)$
defined by  X.-H.~Luo and P.~Zhang is in fact
a full subcategory of the module category over $A\otimes_k kQ$.
They use this category to describe Gorenstein projective modules
over $A\otimes_k kQ$.
Our work is inspired by the main result of \cite{LZ2013}.

Let $A$ be a coherent algebra and
$B$ be a finite-dimensional algebra over a field $k$.
If $B$ is a Gorenstein algebra,
the following result describes
finitely presented Gorenstein projective $A\otimes_k B$-modules
in terms of their underlying onesided modules.
For each left $A\otimes_k B$-module $X$,
we denote by $_AX$  the underlying left $A$-module of $X$
and by $_BX$ the underlying left $B$-module of $X$.
Denote by $D$ the $k$-dual functor $\Hom_k(-,k)$.

\begin{thm} \label{thm:A}
Let $A$ be a coherent algebra
and $B$ be a finite-dimensional Gorenstein algebra over a field $k$.
Let $X$ be a finitely presented left $A\otimes_k B$-module.
Then  $X$ is Gorenstein projective
if and only if $_A(DB\otimes_B X)$ is a Gorenstein projective left $A$-module
and $_BX$ is a Gorenstein projective left $B$-module.
\end{thm}

If $B$ has finite global dimension,
we have  the following more precise description
of finitely presented Gorenstein projective $A\otimes_k B$-modules.
For each left $B$-module $M$,
recall that the radical $\rad_BM$ of $M$
is the intersection of all maximal submodules of $M$;
see, for example, \cite[Chapter~3, \S9]{AF1974}.

\begin{prop} \label{prop:B}
Let $A$ be a coherent algebra
and $B$ be a finite-dimensional algebra of finite global dimension
over a field $k$.
Let $X$ be a finitely presented left $A\otimes_k B$-module.
Then $X$ is Gorenstein projective
if and only if $_A(X/\rad_BX)$ is a Gorenstein projective left $A$-module
and $_BX$ is a projective left $B$-module.
\end{prop}

The paper is organized as follows.
In Section 2, we recall some facts about
the Gorenstein projective modules and
the tensor products of algebras.
The proofs of Theorem 1.1 and Proposition 1.2 are given in
Section 3 and Section 4, respectively.
In Section 5, we give some applications of our results.

\section{Preliminaries} \label{sec:2}
\subsection{Notation}
Throughout this paper,
$k$ is a fixed field.
Let $\otimes$ denote the tensor product over $k$,
and let $D$ denote the $k$-dual functor $\Hom_k(-,k)$.
For each $k$-algebra $A$,
we denote by $A^\op$ the  opposite algebra of $A$.
All modules are left modules;
all categories, morphisms and functors are $k$-linear.

\subsection{Tensor products of algebras}
Let $A$ be a $k$-algebra
and $B$ be a finite-dimensional $k$-algebra.
Recall from \cite[Chapter~IX, Theorem~2.8a]{CE1956} the following.

\begin{lem} \label{lem:adjoint}
Let $M$ be an $A$-module,
$P$ be a finite-dimensional projective $B^\op$-module
and $X$ be an $A\otimes B$-module.
Then for each $n\geq 0$, there is an isomorphism
\[\Ext^n_{A\otimes B}(X,M\otimes DP)\simeq \Ext^n_A(X\otimes_B P,M).\]
\end{lem}

\subsection{Gorenstein projective modules}
Let $A$ be a $k$-algebra.
Following \cite{Buc1987,EJ2000},
a complex of projective $A$-modules
\[P^{\bullet}:\cdots\lto P^{-1} \overset{d^{-1}}\lto P^0\overset{d^0}\lto P^1 \overset{d^{1}}\lto \cdots\]
is {\em totally acyclic}
if $P^{\bullet}$ is acyclic and $\Hom_A(P^\bullet,Q)$ is acyclic for each projective $A$-module $Q$.
An $A$-module $M$ is {\em Gorenstein projective} if there is
a totally acyclic complex $P^{\bullet}$ of projective $A$-modules
such that $M=\Ker d^1$;
the complex $P^\bullet$ is called a {\em complete resolution} of $M$.

From now on, let $A$ be a {\em coherent} $k$-algebra, that is,
a left and right coherent  algebra over $k$.
Denote by $A\-\mod$ the category of finitely presented $A$-modules;
it is an abelian category.
Denote by $A\-\proj$ the full subcategory of $A\-\mod$ formed by finitely generated
projective $A$-modules.
For each Gorenstein projective $A$-module $M$ which is finitely presented,
there is a complete resolution $P^\bullet$ of $M$
such that $P^\bullet$ is a totally acyclic complex of finitely generated
projective $A$-modules;
it is called a complete resolution of $M$ in $A\-\proj$.
For be a complex of finitely generated projective $A$-modules $P^\bullet$,
recall that $P^\bullet$  is totally acyclic if and only if
$P^{\bullet}$ and $\Hom_A(P^\bullet,A)$ are both acyclic.

Denote by $(-)^*$  the $A$-dual functor $\Hom_A(-,A)$ or $\Hom_{A^\op}(-,A)$.
For each $A$-module $M$,
there is an {\em evaluation map} $\ev^A_M: M \to M^{**}$ of $M$ with value in $A$,
where $\ev^A_M(m)(f)=f(m)$ for each $m\in M$ and $f\in M^*$.
Following \cite[Chapter~X, Proposition~10.2.6]{EJ2000},
 a finitely presented $A$-module $M$ is Gorenstein projective
if and only if $M$ satisfies the following conditions:
\begin{enumerate}
  \item $\Ext^n_A(M,A)=0$ for $n\geq 1$;
  \item $\Ext^n_{A^\op}(M^*,A)=0$ for $n\geq 1$;
  \item the evaluation map $\ev^A_M:M \to M^{**}$ is an isomorphism.
\end{enumerate}

Denote by $A\-\Gproj$ the full subcategory of $A\-\mod$ formed by
finitely presented Gorenstein projective $A$-modules.
We will need the following facts about Gorenstein projective $A$-modules;
see, for example, \cite[Lemma 4.2.2]{Buc1987}.
\begin{enumerate}[leftmargin=3em]
  \item[(GP1)] $A\-\Gproj$ is a Frobenius category with projective objects  precisely $A\-\proj$.
  \item[(GP2)] $A\-\Gproj$ is closed under direct summands,
               extensions and kernels of epimorphisms in $A\-\mod$.
  \item[(GP3)] The $A$-dual functors induce exact dualities between $A\-\Gproj$ and $A^\op\-\Gproj$.
\end{enumerate}

Recall that $A\-\Gproj^{\perp}$ contains all finitely presented $A$-modules
of finite projective dimension.
Here, $A\-\Gproj^{\perp}$ denotes the {\em right perpendicular category}
\[A\-\Gproj^{\perp}=\{M\in A\-\mod \mid
\Ext^n_A(G,M)=0, \forall\; G \in A\-\Gproj, \forall\; n \geq1\}\]
of $A\-\Gproj$.
In fact, the full subcategory $A\-\Gproj^{\perp}$ is closed under cokernels
of monomorphisms in $A\-\mod$  and contains $A\-\proj$.

The following lemma is well known.

\begin{lem} \label{lem:excok}
Let $\xi: 0\to M_m \to \cdots \to M_1 \xto{\partial} M_0$
be a sequence of finitely presented Gorenstein projective $A$-modules with $m\geq 1$.
Then the following statements are equivalent.
\begin{enumerate}
  \item $\xi^{*}=\Hom_A(\xi,A)$ is exact.
  \item  $\xi$ is exact and $\Coker{\partial}$ is a Gorenstein projective $A$-module.
\end{enumerate}
\end{lem}

\begin{proof}
``$(1)\Longrightarrow$ (2)"
Since $\xi^{*}$ is exact,
$\Ker\partial^{*}$ is a
finitely presented Gorenstein projective  $A^\op$-module by (GP2).
It follows from (GP3) that  $\xi\simeq (\xi^{*})^{*}$
is exact and $\Coker\partial\simeq (\Ker\partial^{*})^{*}$
is a finitely presented Gorenstein projective  $A^\op$-module.

``$(2)\Longrightarrow (1)$" This follows from (GP3).
\end{proof}

\section{The proof of Theorem~1.1} \label{sec:3}
Let $A$ be a coherent $k$-algebra
and $B$ be a finite-dimensional $k$-algebra.
Then the tensor product $A\otimes B$ of $A$ and $B$
is  a coherent $k$-algebra.
Let $X$ be an $A\otimes B$-module $X$.
Recall that $_AX$ denotes the underlying $A$-module of $X$
and $_BX$ denotes the underlying $B$-module of $X$.
Let $X$ be a finitely presented $A\otimes B$-module.
Then  $_AX$ is  finitely presented, but
$_BX$ is not necessarily finitely presented.
By the equivalent conditions for Gorenstein projective modules and
Lemma \ref{lem:adjoint}, we know that  $_AX$ is Gorenstein projective
if and only if $X$ satisfies the following conditions:
\begin{enumerate}
  \item $\Ext^n_{A\otimes B}(X,A\otimes DB)=0$ for $n\geq 1$;
  \item $\Ext^n_{(A\otimes B)^\op}(X^{\vee},A\otimes DB)=0$ for $n\geq 1$;
  \item the evaluation map $\ev^{A\otimes DB}_X:X\to X^{\vee\vee}$ is an isomorphism.
\end{enumerate}
Here, we denote $(-)^{\vee}=\Hom_{A\otimes B}(-,A\otimes DB)$
or $\Hom_{(A\otimes B)^\op}(-,A\otimes DB)$.

The following result indicates that the duality
between the categories of finite-dimensional modules
over $B$ and $B^\op$ induces some special ``duality''
between the categories of modules over $A\otimes B$ and $(A\otimes B)^\op$.

\begin{lem} \label{lem:hatdual}
Let $U$ be a finite-dimensional $B$-module,
and let $X$ be an $A\otimes B$-module such that
$\Ext^n_{A\otimes B}(X,A\otimes DB)=0$  for $n\geq 1$.
Then for each $n\geq 0$,
there is an isomorphism
\[\Ext^n_{A\otimes B}(X,A\otimes U)\simeq
\Ext^n_{(A\otimes B)^\op}(A\otimes DU, X^{\vee}).\]
\end{lem}

\begin{proof}
We have $(A\otimes U)^{\vee} \simeq A\otimes DU$ and
$(A\otimes DU)^{\vee} \simeq A\otimes DDU$ by \cite[Chapter~XI, Theorem~3.1]{CE1956}.
Since $U$ is finite dimensional over $k$,
$A\otimes DDU$ and $A\otimes U$ are natural isomorphic.
Then by \cite[Chapter~5, Proposition~20.7]{AF1974},
there is an isomorphism
\[\Hom_{A\otimes B}(X,A\otimes U){\simeq} \Hom_{(A\otimes B)^\op}(A\otimes DU, X^{\vee}).\]

Take an exact sequence of finite-dimensional $A$-modules
\begin{equation} \label{eq:seq1}
0 \lto U \to I \overset{f}\lto V \lto 0
\end{equation}
such that $I$ is injective.
Applying $\Hom_{A\otimes B}(X,A\otimes-)$ to \eqref{eq:seq1}
gives rise to an exact sequence
\begin{equation} \label{eq:seq2}
\begin{split}
    &\Hom_{A\otimes B}(X,A\otimes I) \overset{\Hom_{A\otimes B}(X,A\otimes f)}
    \lto \Hom_{A\otimes B}(X,A\otimes V) \\ &\lto
    \Ext^1_{A\otimes B}(X,A\otimes U) \lto \Ext^1_{A\otimes B}(X,A\otimes I).
\end{split}
\end{equation}
Since $I$ is an injective $B$-module, we have $\Ext^1_{A\otimes B}(X,A\otimes I)=0$.

Applying the $k$-dual functor $D$ to \eqref{eq:seq1},
we obtain an exact sequence
\begin{equation} \label{eq:seq3}
    0 \lto DV \overset{Df} \lto DI \lto DU \lto 0.
\end{equation}
Applying $\Hom_{(A\otimes B)^\op}(A\otimes-,X^{\vee})$ to \eqref{eq:seq3} yields an exact sequence
\begin{equation} \label{eq:seq4}
\begin{split}
    &\Hom_{(A\otimes B)^\op}(A\otimes DI,X^\vee)
    \overset{\Hom_{(A\otimes B)^\op}(A\otimes Df,X^\vee)}
    \lto \Hom_{(A\otimes B)^\op}(A\otimes DV,X^\vee)  \\
    &\lto\Ext^1_{(A\otimes B)^\op}(A\otimes DU,X^\vee)
    \lto \Ext^1_{(A\otimes B)^\op}(A\otimes DI,X^\vee).
\end{split}
\end{equation}
Since $A\otimes DI$ is projective, we have
$\Ext^1_{(A\otimes B)^\op}(A\otimes DI,X^\vee)=0$.

Compare the sequences~\eqref{eq:seq2} and~\eqref{eq:seq4}.
Then there is an isomorphism
\[\Ext^1_{A\otimes B}(X,A\otimes U)\simeq
\Ext^1_{(A\otimes B)^\op}(A\otimes DU, X^{\vee}).\]

By dimension shifting,
we can prove the isomorphisms for $n \geq 2$.
\end{proof}

Let $U$ be a finite-dimensional $B$-module.
Assume that the injective dimension $m$ of $U$ is finite.
Then there is an exact sequence of finite-dimensional $B$-modules
\[\xi: I^0 \overset{d^0}\lto I^1 \lto \cdots \lto  I^m  \lto 0\]
such that each $I^i$ is injective and $\Ker{d^0}=U$.

\begin{lem} \label{lem:key}
Let $U$ be a finite-dimensional $B$-module of finite injective dimension,
and let $X$ be a  finitely presented $A\otimes B$-module with
$_AX$  Gorenstein projective.
Then the following statements are equivalent.
\begin{enumerate}
  \item $\Ext^n_{A\otimes B}(X,A\otimes U)=0$  for $n\geq 1$.
  \item $\Hom_{A\otimes B}(X,A\otimes \xi)$ is exact.
  \item $DU\otimes_BX$ is a Gorenstein projective $A$-module and
  $\Tor^B_n(DU,X)=0$ for  $n\geq 1$.
\end{enumerate}
\end{lem}

\begin{proof}
Since $_AX$  is Gorenstein projective, we have
$\Ext^n_{A\otimes B}(X,A\otimes DB)=0$  for  $n\geq 1$.
If $U$ is injective,
it is easy to see that $(1)$, $(2)$ and $(3)$ all hold.
Then we may assume that  $U$ has injective dimension $m\geq 1$.

``$(1) \Longleftrightarrow (2)$''
We prove this by induction on $m$.

If $m=1$, then ${d^0}:I^0 \to I^1$ is surjective.
Apply $\Hom_{A\otimes B}(X,A\otimes -)$ to the short exact sequence
\[0\lto U \lto I^0\overset{d^0}\lto I^1 \lto 0.\]
Then the long exact sequence yields
\[\Ext^1_{A\otimes B}(X,A\otimes U)=\Coker\Hom_{A\otimes B}(X,A\otimes d^0)\] and
$\Ext^n_{A\otimes B}(X,A\otimes U)=0$  for $n\geq 2$.
Therefore,  $\Hom_{A\otimes B}(X,A\otimes \xi)$ is exact if and only if
$\Ext^n_{A\otimes B}(X,A\otimes U)=0$  for $n\geq 1$.

Suppose that $(1)$ and $(2)$ are equivalent for $m=k\geq1$.
If the injective dimension of $U$ is $k+1$,
then the injective dimension of $V=\Coker d^0$ is $k$.
There is a surjective map $p:I^0 \to V$ and
an exact sequence of finite-dimensional $B$-modules
\[\xi': I^1 \overset{d^1}\lto I^2 \lto \cdots \lto  I^m  \lto 0\]
such that each $I^i$ is injective and $\Ker{d^1}=V$.
Observe that $\Hom_{A\otimes B}(X,A\otimes \xi)$ is exact
if and only if $\Hom_{A\otimes B}(X,A\otimes p)$ is surjective and
$\Hom_{A\otimes B}(X,A\otimes \xi')$ is exact.

Apply $\Hom_{A\otimes B}(X,A\otimes -)$ to the short exact sequence
\[0\lto U \lto I^0\overset{p}\lto V \lto 0.\]
It follows from the long exact sequence that we have
\[\Ext^1_{A\otimes B}(X,A\otimes U)=\Coker\Hom_{A\otimes B}(X,A\otimes p)\]
and
\[\Ext^{n+1}_{A\otimes B}(X,A\otimes U)=\Ext^n_{A\otimes B}(X,A\otimes V)  \;\mbox{ for } n\geq 1.\]
By induction hypothesis, we know that
$\Hom_{A\otimes B}(X,A\otimes \xi')$ is exact if and only if
$\Ext^n_{A\otimes B}(X,A\otimes V)=0$  for $n\geq 1$.
Therefore,
$\Hom_{A\otimes B}(X,A\otimes \xi)$ is exact if and only if
$\Ext^n_{A\otimes B}(X,A\otimes U)=0$  for $n\geq 1$.

``$(2) \Longleftrightarrow (3)$''
It follows from Lemma \ref{lem:adjoint} that
$\Hom_{A\otimes B}(X,A\otimes \xi)$ is isomorphic to
$\Hom_A(D\xi\otimes_B X,A)$.
In particular,
$\Hom_{A\otimes B}(X,A\otimes \xi)$ is exact
if and only if $\Hom_A(D\xi\otimes_B X,A)$ is exact.
Since $\xi$ is a deleted  injective resolution of $U$,
we have that $D\xi$ is a deleted projective resolution of $DU$.
Observe that $D\xi\otimes_B X$ is a sequence of finitely presented Gorenstein projective $A$-modules.
By Lemma \ref{lem:excok}, $\Hom_A(D\xi\otimes_B X,A)$ is exact
if and only if $D\xi\otimes_B X$ is exact
and $\Coker (D(d^0)\otimes_B X)$ is a Gorenstein projective $A$-module.
A direct calculation shows $\Coker (D(d^0)\otimes_B X) =DU\otimes_BX$
and the homology $H_n(D\xi\otimes_B X)=\Tor^B_n(DU,X)$
for $n\geq 1$.
This finishes our proof.
\end{proof}

Let $A$ be a coherent $k$-algebra
and $B$ be a finite-dimensional $k$-algebra.
For a finitely presented Gorenstein projective $A\otimes B$-module $X$,
is $_AX$ a Gorenstein projective $A$-module?
The following lemma gives an affirmative answer
if the projective dimension of $_BDB$ is finite.
Here, we recall that the right perpendicular category  $A\otimes B\-\Gproj^{\perp}$
of $A\otimes B\-\Gproj$ contains
all finitely presented $A\otimes B$-modules of finite projective dimension.

\begin{lem} \label{lem:gprojperp}
Let $X$ be a finitely presented Gorenstein projective $A\otimes B$-module and
$P^\bullet$ be a complete resolution of $X$ in $A\-\proj$.
Assume that the projective dimension of $_BDB$ is finite.
Then the following statements hold.
\begin{enumerate}
  \item $\Ext^n_{A\otimes B}(X,A\otimes DB)= 0$ for $n\geq 1$.
  \item $\Hom_{A\otimes B}(P^\bullet,A\otimes DB)$ is acyclic.
  \item  $_AX$ is a Gorenstein projective $A$-module.
\end{enumerate}
\end{lem}

\begin{proof}
$(1)$ Since the projective dimension of $_BDB$ is finite,
 the projective dimension of  $_{A\otimes B}A\otimes DB$ is also finite.
Then $A\otimes DB$ belongs to $A\otimes B\-\Gproj^\perp$.
Therefore, we have
$\Ext^n_{A\otimes B}(X,A\otimes DB)= 0$ for $n\geq 1$.

$(2)$ Since each cocycle of $P^\bullet$
is a finitely presented  Gorenstein projective $A\otimes B$-module,
it follows from $(1)$ that $\Hom_{A\otimes B}(P^\bullet,A\otimes DB)$ is acyclic.

$(3)$ Observe that  $_AP^\bullet$ is a complex of finitely generated projective $A$-modules.
By Lemma \ref{lem:adjoint}, we infer that $\Hom_A(P^\bullet,A)$ is isomorphic to
$\Hom_{A\otimes B}(P^\bullet,A\otimes DB)$, which is acyclic by (2).
Then  $_AP^\bullet$ is a complete resolution of $_AX$.
Therefore, $_AX$ is a Gorenstein projective $A$-module.
\end{proof}

Let $A$ be a coherent $k$-algebra.
We recall that the category of finitely presented Gorenstein projective
$A$-modules is closed under direct summands,
extensions and kernels of epimorphisms.
Recall that
a finite-dimensional $k$-algebra $B$ is called
a {\em Gorenstein algebra} if the injective dimensions
of $_BB$ and $B_B$ are both finite.
In this case,
the projective dimension of $_BDB$ is finite.
Recall from \cite[Proposition~3.10]{Bel2005} that
a $B$-module $M$ is Gorenstein projective
if and only if $\Tor^B_n(DB,M)=0$ for $n\geq 1$.

\begin{thm} \label{thm-mainthm}
Let $A$ be a coherent $k$-algebra and
$B$ be a finite-dimensional Gorenstein $k$-algebra.
Let $X$ be a finitely presented $A\otimes B$-module.
Then the following statements are equivalent.
\begin{enumerate}
  \item $X$ is a Gorenstein projective $A\otimes B$-module.
  \item $_AX$ is a Gorenstein projective $A$-module and $\Ext^{n}_{A\otimes B}(X,A\otimes B)=0$ for $n\geq 1$.
  \item $_A(DB\otimes_B X)$ is a Gorenstein projective $A$-module and $_BX$ is a Gorenstein projective $B$-module.
\end{enumerate}
\end{thm}

\begin{proof}
``$(1)\Longrightarrow (2)$"
Since  $X$ is a Gorenstein projective $A\otimes B$-module,
we have $\Ext^n_{A\otimes B}(X,A\otimes B)=0$ for $n\geq 1$.
Then $_AX$ is a Gorenstein projective $A$-module by Lemma \ref{lem:gprojperp}.

``$(2)\Longrightarrow (1)$"
Since $B$ is a Gorenstein algebra,
it is easy to see that $A\otimes DB$ is a tilting $(A\otimes B)^\op$-module
in the sense of \cite{CPS1986,Miy1986}.
Since $\Ext^{n}_{A\otimes B}(X,A\otimes B)=0$ for   $n\geq 1$,
we have $\Ext^n_{(A\otimes B)^\op}(A\otimes DB, X^{\vee})=0$ for $n\geq 1$ by Lemma \ref{lem:hatdual}.
Here, we recall that $(-)^{\vee}$ denotes the $A\otimes DB$-dual functors.

It follows from \cite[Theorem~1.16]{Miy1986}
and \cite[Proposition~1.20]{Miy1986} that for each $n\geq 0$,
there is an isomorphism
\begin{equation*}
\begin{split}
    &\Ext^n_{(A\otimes B)^\op}(X^{\vee},A\otimes DB) \\
    \simeq & \Ext^n_{(A\otimes B)^\op}(\Hom_{(A\otimes B)^\op}(A\otimes DB,X^{\vee}),
    \Hom_{(A\otimes B)^\op}(A\otimes DB,A\otimes DB)).
\end{split}
\end{equation*}
Then for each $n\geq 1$, we have
\begin{equation*}
\begin{split}
    &\Ext^n_{(A\otimes B)^\op}(\Hom_{A\otimes B}(X,A\otimes B),A\otimes B) \\
    \simeq & \Ext^n_{(A\otimes B)^\op}(\Hom_{A\otimes B}(X,A\otimes B),
    \Hom_{A\otimes B}(A\otimes B,A\otimes B)) \\
    \simeq & \Ext^n_{(A\otimes B)^\op}(\Hom_{(A\otimes B)^\op}(A\otimes DB,X^{\vee}),
    \Hom_{(A\otimes B)^\op}(A\otimes DB,A\otimes DB)) \\
    \simeq & \Ext^n_{(A\otimes B)^\op}(X^{\vee},A\otimes DB) \\
    =  &0.
\end{split}
\end{equation*}

Observe that the composite of $X \to X^{\vee\vee}$ with the series of isomorphisms
\begin{equation*}
\begin{split}
     X^{\vee\vee}=      & \Hom_{(A\otimes B)^\op}(X^{\vee},A\otimes DB) \\
    \simeq & \Hom_{(A\otimes B)^\op}(\Hom_{(A\otimes B)^\op}(A\otimes DB,X^{\vee}),
    \Hom_{(A\otimes B)^\op}(A\otimes DB,A\otimes DB)) \\
    \simeq & \Hom_{(A\otimes B)^\op}(\Hom_{A\otimes B}(X,A\otimes B),
    \Hom_{A\otimes B}(A\otimes B,A\otimes B)) \\
    \simeq & \Hom_{(A\otimes B)^\op}(\Hom_{A\otimes B}(X,A\otimes B),A\otimes B) \\
\end{split}
\end{equation*}
is an isomorphism.
It is easy to see that the composite  is
the evaluation map of $X$ with value in $A\otimes B$.
Therefore,
$X$ is a Gorenstein projective $A\otimes B$-module.

``$(2)\Longrightarrow(3)$"
Since $_AX$ is a Gorenstein projective $A$-module
and $\Ext^{n}_{A\otimes B}(X,A\otimes B)=0$ for $n\geq 1$,
we infer that $_A(DB\otimes_B X)$ is a Gorenstein projective $A$-module
and $\Tor^B_n(DB,M)=0$ for $n\geq 1$ by Lemma \ref{lem:key}.
Then $_BX$ is a Gorenstein projective $B$-module.

``$(3)\Longrightarrow(2)$"
Since $_BX$ is Gorenstein projective,
we have $\Tor^B_n(DB,M)=0$ for  $n\geq 1$.
Observe that $_A(DB\otimes_B X)$ is a Gorenstein projective $A$-module and
$B_B$ has finite injective dimension.
Then  $_AX$ is a Gorenstein projective $A$-module.
It follows from Lemma \ref{lem:key} that
$\Ext^{n}_{A\otimes B}(X,A\otimes B)=0$ for $n\geq 1$.
\end{proof}

Recall that a finite-dimensional  $k$-algebra $B$ is called
a {\em self-injective} algebra if
$_BB$ is injective or, equivalently, $_BDB$ is projective.
The following corollary is also a special case of \cite[Theorem~3.6]{Che2013}.

\begin{cor} \label{cor:mainthm}
Let $A$ be a coherent $k$-algebra
and  $B$ be a finite-dimensional self-injective $k$-algebra.
Then a finitely presented $A\otimes B$-module $X$ is Gorenstein projective
if and only if $_AX$ is a Gorenstein projective $A$-module.
\end{cor}

\begin{proof}
``$\Longrightarrow$" This follows from Lemma \ref{lem:gprojperp}.

``$\Longleftarrow$" Since $_BDB$ is a projective $B$-module
and $_AX$ is a Gorenstein projective $A$-module,
we infer that $_A(DB\otimes_B X)$ is a Gorenstein projective $A$-module.
Observe that all $B$-modules are Gorenstein projective.
Then it follows from Theorem \ref{thm-mainthm}
that $X$ is a Gorenstein projective  $A\otimes B$-module.
\end{proof}

\section{The proof of Proposition 1.2} \label{sec:4}
Throughout this section,
$A$ is a coherent $k$-algebra and
$B$ is a finite-dimensional $k$-algebra of finite global dimension.
Let $M$ be a $B$-module.
Recall that the radical $\rad_BM$ of $M$ is the intersection of all maximal submodules of $M$.
In particular,
the radical $\rad B$ of $_BB$ is an ideal of $B$.
Recall that $D$ denotes the $k$-dual functor.

\begin{lem} \label{lem:fingldim}
Let $X$ be a finitely presented $A\otimes B$-module.
Then the following statements are equivalent.
\begin{enumerate}
  \item $\Ext^n_{A\otimes B}(X,A\otimes B)=0$ for $n\geq 1$.
  \item $\Ext^n_{A\otimes B}(X,A\otimes S)=0$ for each simple $B$-module $S$ and $n\geq 1$.
  \item $\Ext^n_{A\otimes B}(X,A\otimes D(B/\rad B))=0$ for $n\geq 1$.
\end{enumerate}
\end{lem}

\begin{proof}
Let $\X$ be the full subcategory
of $B\-\mod$ formed by objects $U$ which satisfies
$\Ext^n_{A\otimes B}(X,A\otimes U)=0$ for $n\geq 1$.
Then $\X$ is closed under direct summands,
extensions and cokernels of monomorphisms in $B\-\mod$.

``$(1) \Longrightarrow (2)$''
Observe that simple $B$-modules have finite projective dimensions.
Since $B$ belongs to $\X$,
each simple $B$-module $S$ belongs to $\X$.

``$(2) \Longrightarrow (1)$''
Since $_BB$ has finite length and each simple $B$-module $S$ belongs to $\X$,
we infer that $B$ belongs to $\X$.

``$(2) \Longleftrightarrow (3)$''
Since $_BD(B/\rad B)$ is a finite-dimensional semisimple $B$-module
and each simple $B$-module $S$ is a direct summand of $_BD(B/\rad B)$,
it follows that $D(B/\rad B)$ belongs to $\X$
if and only if each simple $B$-module $S$ belongs to $\X$.
\end{proof}

Let $X$ be an $A\otimes B$-module.
Then $\rad B\cdot X=\rad_BX$ is an $A\otimes B$-submodule of $X$;
see \cite[Chapter~4, Corollary~15.21]{AF1974}.
Observe that $(B/\rad B)\otimes_BX$ is isomorphic to $X/(\rad B\cdot X)$.
Then  $X/\rad_BX \simeq (B/\rad B)\otimes_BX$ as $A\otimes B$-modules.

\begin{prop} \label{prop:fingldim}
Let $A$ be a coherent $k$-algebra and
$B$ be a finite-dimensional $k$-algebra of finite global dimension.
Let $X$ be a finitely presented $A\otimes B$-module.
Then $X$ is Gorenstein projective
if and only if $_A(X/\rad_B X)$ is a Gorenstein projective $A$-module
and $_BX$ is a  projective $B$-module.
\end{prop}

\begin{proof}
``$\Longrightarrow$"
Since $X$ is a Gorenstein projective $A\otimes B$-module,
it follows from Theorem \ref{thm-mainthm} and Lemma \ref{lem:fingldim}
that $_AX$ is a  Gorenstein projective $A$-module
and $\Ext^n_{A\otimes B}(X,A\otimes D(B/\rad B))=0$
for $n\geq 1$.

Since  $_AX$ is Gorenstein projective
and the  injective dimension of $_BD(B/\rad B)$ is finite,
$_A(X/\rad_B X)$ is a Gorenstein projective $A$-module
and $\Tor^B_n(B/\rad B,X)=0$ for $n\geq 1$ by Lemma \ref{lem:key}.
Then $_BX$ is flat.
Since $B$ is a finite-dimensional $k$-algebra,
it follows that $_BX$ is a projective $B$-module.

``$\Longleftarrow$"
Since $_BX$  projective,
we have $\Tor^B_n(S,X)=0$ for  each simple $B^\op$-modules $S$ and $n\geq 1$.
Since $(B/\rad B)\otimes X$ is a Gorenstein projective $A$-module,
$S\otimes _BX$ is a Gorenstein projective $A$-module
for each simple $B^\op$-modules $S$.
Since $B_B$ has finite length,
we infer that $_AX$ is a Gorenstein projective $A$-module.

Since $_BX$ is projective,
$\Tor^B_n(B/\rad B,X)=0$ for $n\geq 1$.
Observe that $_AX$ and $_A(X/\rad_B X)$ are Gorenstein projective  $A$-modules,
and  $_BD(B/\rad B)$ has finite injective dimension.
By Lemma \ref{lem:key}, we have $\Ext^n_{A\otimes B}(X,A\otimes D(B/\rad B))=0$
for $n\geq 1$.
Then it follows from Theorem \ref{thm-mainthm}
and Lemma \ref{lem:fingldim} that
$X$ is  a Gorenstein projective  $A\otimes B$-module.
\end{proof}

\begin{exm} \label{ex:quiver}
Let $Q=(Q_0,Q_1)$ be a finite quiver.
Here, $Q_0$ is the set of vertices and $Q_1$ is the set of arrows of the quiver $Q$;
see \cite[Chapter~III, \S1]{ARS1995}.

Let $I$ be an admissible ideal of the path algebra $kQ$ such that
$B=kQ/I$ be a finite-dimensional $k$-algebra of finite global dimension.
We denote by $e_i$ the trivial path, $S(i)$ the simple $B$-module and $I(i)$
the indecomposable injective $B$-module
at the vertex $i\in Q_0$.

Let $A$ be a coherent $k$-algebra
and $X$ be a finitely presented $A\otimes B$-module.
Then for each $i\in Q_0$ there is a finitely presented $A$-module $X_i=e_iX$;
and for each $\alpha\in Q_1$ there an $A$-module map $X_\alpha:X_{s(\alpha)} \to X_{t(\alpha)}$.
Here, $s(\alpha)$ is the staring vertex   of $\alpha$ and
$t(\alpha)$ is the terminating vertex of $\alpha$.

For each $\alpha\in Q_1$, let $f_i$ be the natural $A$-module map
\[f_i=(X_\alpha): \oplus_{\alpha\in Q_1,t(\alpha)=i}\;X_{s(\alpha)} \to X_i.\]
We claim that there is an isomorphism $DS(i)\otimes_B X\simeq \Coker f_i$ of $A$-modules.
Then it follows from Proposition \ref{prop:fingldim} that
$X$ is Gorenstein projective if and only if $_BX$ is a projective $B$-module
and $\Coker f_i$ is a Gorenstein projective $A$-modules for each $i\in Q_0$;
compare \cite{LZ2013,LZ2015}.

For the claim, we take the injective copresentation
\[0 \lto S(i)\lto  I(i) \overset{(h_\alpha)}\lto \oplus_{\alpha\in Q_1,t(\alpha)=i}\;I(s\alpha)\]
of the simple $B$-module $S(i)$,
where $h_\alpha$ is induced by multiplication of $\alpha$.
Applying the left exact functor $D(-)\otimes_BX$ to the above sequence,
we obtain an isomorphism $DS(i)\otimes_B X\simeq \Coker f_i$ of $A$-modules.
\end{exm}

\section{Applications} \label{sec:5}
Let $A$ be a finite-dimensional $k$-algebra.
Let $n\geq 1$ be a positive integer,
and denote by $[n]$ the set
$\{1,2,\cdots,n\}$ of positive integers less than or equal to $n$.
Recall from \cite[7.1]{PX1997} that an {\em $n$-periodic complex} of finite-dimensional $A$-modules
is a collection $(X_i,X_{\alpha_i},i\in[n])$,
where $X_i$ is a finite-dimensional $A$-module
and $X_{\alpha_i}$ is an $A$-module map
from $X_i$ to $X_{i+1}$ satisfying
$X_{\alpha_{(i+1)}}X_{\alpha_i}=0$ for each $i\in [n]$.
Here, we identify $n+1$ with $1$.
An {\em $n$-periodic chain map} is a collection $(f_i,i\in[n])$,
where $f_i:X_i\to Y_i$  is an $A$-module map satisfying
$f_{i+1}X_{\alpha_i}=Y_{\alpha_i}f_i$ for each $i\in [n]$.
Denote by $\bfC_n(A\-\mod)$ the category of $n$-periodic complexes of
finite-dimensional $A$-modules.

Let $\bbZ_n$ be the quiver with $n$ vertices and $n$ arrows
which forms an oriented cycle.
The vertex set of $\bbZ_n$ is identified with $[n]=\{1,2,\cdots,n\}$;
and there is a unique arrow $\alpha_i$ from $i$ to $i+1$ for each $i \in [n]$.
Let $B_n$ be the radical square zero $k$-algebra given by $\bbZ_n$.
Then $B_n$ is a self-injective Nakayama algebra.
Denote by $e_i$ the trivial path at $i\in[n]$.
Recall that $A\otimes B_n\-\mod$ denotes the category of finite-dimensional $A\otimes B_n$-modules.

There is an equivalence $R:A\otimes B_n\-\mod\to\bfC_n(A\-\mod)$
of abelian categories; compare \cite[Lemma 2.1]{LZ2013}.
The functor $R$ sends a module $X$ in $A\otimes B_n\-\mod$ to $(X_i,X_{\alpha_i},i\in[n])$,
where $X_i=e_iX$ and  $X_{\alpha_i}$ is induced by multiplication of $\alpha_i$.

Denote by $\bfC_n(A\-\proj)$ the category of $n$-periodic complexes of
finite-dimensional projective $A$-modules;
it is a Frobenius category \cite[Proposition 7.1]{PX1997}.
Denote by $\bfK_n(A\-\proj)$ the homotopy category of complexes in $\bfC_n(A\-\proj)$;
it is a triangulated category \cite{Hap1987}.
Note that a morphism in $\bfC_n(A\-\proj)$ is homotopic to zero
if and only if it factors through a contractible $n$-periodic complex of projective $A$-modules.

Recall that $A\otimes B_n\-\Gproj$ denotes the Frobenius category of
finite-dimensional Gorenstein projective $A\otimes B_n$-modules.
Denote by $A\otimes B_n\-\uGproj$ the stable category of $A\otimes B_n\-\Gproj$;
it is a triangulated category \cite{Hap1987}.
The stable category $A\otimes B_n\-\uGproj$
is obtained from $A\otimes B_n\-\Gproj$ by factoring out the ideal
of all maps which factor through projective $A\otimes B_n$-modules;
see \cite[Chapter~IV, \S1]{ARS1995}.

\begin{lem} \label{lem:percom}
Let $A$ be a finite-dimensional $k$-algebra of finite global dimension.
Then the functor $R$ induces equivalences:
\begin{enumerate}
  \item $\bfC_n(A\-\proj)\simeq A\otimes B_n\-\Gproj$ of Frobenius categories;
  \item $\bfK_n(A\-\proj)\simeq A\otimes B_n\-\uGproj$ of triangulated categories.
\end{enumerate}

\end{lem}

\begin{proof}
$(1)$ Since $A$ has finite global dimension,
any Gorenstein projective $A$-module is projective.
By Corollary \ref{cor:mainthm},
a finite-dimensional $A\otimes B_n$-module $X$ is Gorenstein projective
if and only if $X_i$ is a projective $A$-module for each $i\in [n]$.
It is routine to check that $R$ preserves exact structures.
Then $R$ induces an equivalence $\bfC_n(A\-\proj) \simeq A\otimes B_n\-\Gproj$
of Frobenius categories.

$(2)$ Observe that a finite-dimensional $A\otimes B_n$-module $X$ is projective
if and only if $R(X)$ is a contractible $n$-periodic complex of projective $A$-modules.
Then it follows from $(1)$ that
$R$  induces an equivalence $\bfK_n(A\-\proj) \simeq A\otimes B_n\-\uGproj$
of triangulated categories.
\end{proof}

We now give a new proof of the main theorem in \cite{Zha2014}.

\begin{prop} \label{prop:percom}
Let $A$ and $A'$ be two finite-dimensional $k$-algebras of finite global dimension.
If $A$ and $A'$ are derived equivalent,
then $\bfK_n(A\-\proj)$ and $\bfK_n(A'\-\proj)$ are triangle equivalent.
\end{prop}

\begin{proof}
Since $A$ and $A'$ are derived equivalent,
it follows from \cite[Theorem 2.1]{Ric1991} that
$A\otimes B_n$ and  $A'\otimes B_n$ are also derived equivalent.
Then the triangulated quotient categories
$\bfD^\bo(A\otimes B_n\-\mod)/\bfK^\bo(A\otimes B_n\-\proj)$ and
$\bfD^\bo(A'\otimes B_n\-\mod)/\bfK^\bo(A'\otimes B_n\-\proj)$
are triangle equivalent by \cite[Theorem 1.1]{Ric1989b}.

Since $A$ has finite global dimension and $B_n$ is a self-injective algebra,
$A\otimes B_n$ is a  Gorenstein algebra.
By \cite[Theorem~4.4.1]{Buc1987}, the triangulated quotient category
$\bfD^\bo(A\otimes B_n\-\mod)/\bfK^\bo(A\otimes B_n\-\proj)$ and
$A\otimes B_n\-\uGproj$ are triangle equivalent.
It follows that $A\otimes B_n\-\uGproj$ and $A'\otimes B_n\-\uGproj$ are triangle equivalent.
Then by Lemma \ref{lem:percom}, we know that
$\bfK_n(A\-\proj)$ and $\bfK_n(A'\-\proj)$ are triangle equivalent.
\end{proof}

Let $A$ be a finite-dimensional $k$-algebra.
Denote by $A^\mathrm{e}=A\otimes A^\op$ the enveloping algebra of $A$.
For a finite-dimensional Gorenstein $k$-algebra $A$,
recall that a finite-dimensional  Gorenstein projective $A$-module
of finite projective dimension is projective; see, for example, \cite[Lemma~5.1.1]{Buc1987}.

\begin{prop} \label{prop:reggproj}
Let $A$ be a finite-dimensional Gorenstein $k$-algebra.
Then $A$ is a Gorenstein projective $A^\mathrm{e}$-module
if and only if $A$ is a self-injective algebra.
\end{prop}

\begin{proof}
``$\Longrightarrow$'' Since $A$ is a finite-dimensional Gorenstein projective $A^\mathrm{e}$-module,
it follows from Theorem \ref{thm-mainthm} that $_ADA$ is a Gorenstein projective $A$-module.
Since $_ADA$ has finite projective dimension,
we know that $_ADA$ is projective.
Then $A$ is a self-injective algebra.

``$\Longleftarrow$'' Since $A$ is a self-injective algebra,
$A^\mathrm{e}$ is a self-injective algebra by \cite[Proposition~2.2]{AR1991}.
Then all $A^\mathrm{e}$-modules are Gorenstein projective.
Therefore, $A$ is a Gorenstein projective $A^\mathrm{e}$-module.
\end{proof}

\section*{Acknowledgments}
The author is thankful to Professors Xiao-Wu~Chen,
Yu~Ye and Guodong~Zhou
for numerous inspiring discussions.
The author is supported by China Postdoctoral Science Foundation (No. 2015M581563)
and by STCSM (No. 13DZ2260400).

\end{document}